\documentclass[10pt,a4paper]{amsart}
\usepackage[T1]{fontenc}
\usepackage{fullpage} 
\usepackage{amsmath}
\usepackage{amsthm}
\usepackage{amssymb}
\usepackage{mathtools}
\usepackage[only,mapsfrom]{stmaryrd}
\usepackage{graphicx}
\usepackage{cancel} 
\usepackage{color} 
\usepackage{nccmath}
\usepackage[all]{xy} 
\usepackage{tikz}
\usetikzlibrary{arrows,calc}
\usetikzlibrary{cd}
\usepackage{nicefrac}
\usepackage{enumitem}
\usepackage[colorlinks,citecolor=blue,linkcolor=blue,urlcolor=blue,filecolor=blue,breaklinks]{hyperref}
\usepackage{caption}
\usepackage{comment}
\usepackage{tensor}

\numberwithin{equation}{section}
\numberwithin{figure}{section}

\newtheorem{thm}{Theorem}[section]
\newtheorem{athm}{Theorem}

\newtheorem{lem}[thm]{Lemma}
\newtheorem{prop}[thm]{Proposition}
\newtheorem{cor}[thm]{Corollary}

\newtheorem{ques}[thm]{Question}

\newtheorem*{thm*}{Theorem}
\newtheorem*{conj*}{Conjecture}
\newtheorem*{cor*}{Corollary}
\newtheorem*{ques*}{Question}
\newtheorem*{claim*}{Claim}

\theoremstyle{definition}
\newtheorem{rem}[thm]{Remark}
\newtheorem{defn}[thm]{Definition}

\newtheorem*{rem*}{Remark}

\newcommand{\bQ}{\mathbb{Q}}

\newcommand{\bZ}{\mathbb{Z}}

\newcommand{\Map}{\mathrm{Map}}
\newcommand{\PMap}{\mathrm{PMap}}
\newcommand{\Homeo}{\mathrm{Homeo}}

\newcommand{\RFRS}{\mathrm{RFRS}}
\newcommand{\RFRp}{\mathrm{RFR}\mathnormal{p}}

\newcommand{\PB}{\mathrm{PB}}

\newcommand{\Qrad}{\operatorname{rad}_{\bQ}}
\newcommand{\ab}{\mathrm{ab}}

\newcommand{\incl}[3][right]%
{%
\draw[<-,>=#1 hook] #2 to ($ #2!0.5!#3 $);
\draw[->] ($ #2!0.5!#3 $) to #3;%
}
\newcommand{\inclusion}[5][right]%
{%
\draw[<-,>=#1 hook] #4 to ($ #4!0.5!#5 $) node[#2,font=\small]{#3};
\draw[->,>=stealth'] ($ #4!0.5!#5 $) to #5;%
}

\renewcommand{\geq}{\geqslant}
\renewcommand{\leq}{\leqslant}

\title{Pure braid groups are $\RFRS$}

\author{Xiaolei Wu}
\address{Shanghai Center for Mathematical Sciences, Jiangwan Campus, Fudan University, No. 2005 Songhu Road, Shanghai, 200438, P.R. China}
\email{xiaoleiwu@fudan.edu.cn}

\author{Shengkui Ye}
\address{NYU Shanghai, No. 567 Yangsi West 
  Rd, Pudong New Area, Shanghai, 200124, P.R. China \\
NYU-ECNU Institute of Mathematical Sciences at NYU Shanghai, 3663 Zhongshan Road North, Shanghai, 200062, China}
\email{sy55@nyu.edu}

\subjclass[2020]{20F36, 20F65, 57K20}
\keywords{Pure braid groups, $\RFRS$ groups, $\RFRp$ groups, finite covers, planar surfaces, Fadell--Neuwirth sequence}
\date{August 2026}

\begin{document}

\begin{abstract}
Agol in his 2014 ICM proceedings article \cite[Question 11]{Agol14} asks whether braid groups are (virtually) $\RFRS$. We answer this positively by showing that pure braid groups are $\RFRS$. As a consequence, several families of Artin groups are virtually $\RFRS$, including those of type $A_n$ (the braid groups), $B_n=C_n$,  $\widetilde A_n$, and $\widetilde C_n$.  Our results also provide evidence toward the  problem of whether braid groups, and more generally Artin groups, are virtually special; see  \cite[Problem 9.4]{HagWi10}, \cite[Problem 13.4]{Wise14}.

\end{abstract}

\maketitle

\section*{Introduction}

Residual properties of groups play a central role in understanding their structure and applications to topology and geometry. Among these, the class of \emph{residually finite rationally solvable} ($\RFRS$) groups, introduced by Agol \cite{Agol08}, has been particularly influential due to its connections with virtual properties of $3$-manifolds \cite{Agol14} and fibering \cite{Kielak20}. Recall that a group $G$ is called $\RFRS$ if there is a descending chain $G=G_0 \geq G_1 \geq G_2 \geq \cdots$
of finite-index normal subgroups of $G$ with trivial intersection, such that for
each $i$ the subgroup $G_{i+1}$ contains the kernel of $
G_i \longrightarrow H_1(G_i;\mathbb Q)$.

The braid groups $B_n$ and their subgroups, the pure braid groups $PB_n$, are fundamental objects in algebraic and geometric topology. It is known that braid groups are linear \cite{Bigelow01, Krammer02}  and left-orderable \cite{Dehornoy94}, while pure braid groups are poly-free and residually torsion-free nilpotent \cite{FalkRan88}. Agol asks in \cite[Question 11]{Agol14}  whether braid groups are $\RFRS$. This cannot be true since $\RFRS$ groups are locally indicable \cite[Proposition 2.7]{Fisher24}, but the commutator subgroups of braid groups $B_n$ are finitely generated and perfect once $n\geq 5$ \cite[Theorem 2.1]{GorinLin69}. In \cite[\S 7]{WuYe2025}, we  modified his question to the following:

\begin{ques}\label{ques:brd-rfrs}
    Are braid groups virtually $\RFRS$?
\end{ques}

Our main result answers this modified question affirmatively.

\begin{athm} \label{athm-rfrs}  
    The pure braid groups are $\RFRS$. In fact, they are $\RFRp$ for any prime $p$.
\end{athm}

See  \cite{KoberdaSuciu2020} and \S \ref{sec:par-RFRS} for more information about $\RFRp$ groups. A related question of Wise asks whether braid groups are virtually special \cite[Problem 13.4]{Wise14}.  More generally,  Haglund and Wise ask in \cite[Problem 9.4]{HagWi10} whether all Artin groups are virtually special. Since virtually special groups are virtually $\RFRS$, it makes sense to also ask the following question.

\begin{ques}
    Are Artin groups virtually $\RFRS$?
\end{ques}

We have the following partial results.

\begin{cor}\label{cor-artin-rfrs}
    The Artin groups of type $A_n$ (these are the braid groups), $B_n=C_n$, $I_2(n)$, $\tilde{A}_n, \tilde{C}_n$  are virtually $\RFRp$ for any prime  $p$.
\end{cor}

Note that mapping class groups of surfaces are not virtually $\RFRS$ in general. In fact,  since   fibered 3-manifold groups embed into punctured mapping class groups and RFRSness passes to subgroups, it suffices to find a fibered closed $3$-manifold whose fundamental group is not virtually $\RFRS$. For example, one can take the mapping torus of a Dehn twist on a closed surface of genus at least $2$; its fundamental group is not virtually $\RFRS$ by \cite[Theorem 1.1]{Liu13} and \cite[Theorem 3.7]{MaBe96}.

Our proof of Theorem \ref{athm-rfrs}  uses the following planar lifting of curves which might be of independent interest. Recall that a compact surface is \emph{planar} if it has genus $0$  and nonempty boundary.

\begin{athm}\label{athm:maintower}
Let $S$ be a compact connected oriented planar surface with nonempty boundary, and let $1\ne f\in\pi_1(S)$.  For any prime $p$,  there is a finite tower of arc-cyclic $p$-covers 
\[
        S_r\longrightarrow S_{r-1}\longrightarrow \cdots\longrightarrow S_0=S
\]
such that $f$ lifts to closed loops throughout the whole tower, $S_i$ is planar  for all $i$, and the final lift $f_r\in\pi_1(S_r)$ has nonzero image in $H_1(S_r;\mathbb{Z}/p)$.
\end{athm}

The study of curve lifting has a long history.  Scott \cite{Sc78} proves that any curve on a compact surface can be lifted to a simple closed curve in some finite cover.  Malestein--Putman \cite[Lemma 2.1]{MP10} proves that for any non-nullhomotopic closed curve on an orientable compact surface, there is a degree-8 normal cover such that either the curve is not liftable or it has a lift with a smaller geometric intersection number. Since a surface group $\pi_1(S)$ (where $S$ is not the real projective plane, the Klein
bottle, nor the non-orientable surface of Euler characteristic $-1$) embeds as a subgroup into a right-angled Artin group (RAAG) and a RAAG has $\RFRp$ (see \cite[Theorem 3]{MR2077673}, \cite[Proposition 1.1]{KoberdaSuciu2020}), any curve on the compact surface $S$ can be lifted to a homologically non-trivial closed curve in the final surface of some tower of $\mathbb{Z}/p$-covers (see Corollary \ref{curves}).
Theorem \ref{athm:maintower} says that for planar surfaces such a lifting can be realized in arc-cyclic covers (see \S \ref{sec:cyc-dul-cov} for the definition), and the homological groups can be chosen to be with $\mathbb{Z}/p$-coefficients. 

The proof of Theorem \ref{athm-rfrs} proceeds via a combination of geometric unwrapping in planar surfaces and algebraic control of isotopical actions. In \S~\ref{sec:par-RFRS}, we recall the criterion that a finitely generated group is $\RFRS$ (or $\mathrm{RFR}p$) if and only if every nontrivial element is detected by a finite partial tower of subgroups with prescribed rational (or mod~$p$) homology quotients. To exploit this criterion for pure braid groups, we rely on the Fadell--Neuwirth splitting
\[
PB_n \cong F_{n-1}\rtimes PB_{n-1},
\]
where $F_{n-1}$ is identified with the fundamental group of a compact planar surface $S_{n-1}$. The heart of the proof is a purely topological unwrapping result (Theorem \ref{athm:maintower}). This is established in \S\ref{sec:cover} by successively resolving self-intersections of a minimal immersed representative via carefully chosen arcs dual to an innermost essential subloop. Alongside this geometric construction, we must ensure that the action of $PB_{n-1}$ on $S_{n-1}$ is compatible with the tower of covers. This compatibility is achieved in \S\ref{sec:b-contr} through a boundary control argument: for any isotopical action on a planar surface and any arc-cyclic cover, passing to a suitable finite-index subgroup of the acting group guarantees that canonical lifts fix every boundary component of the cover pointwise, thereby inducing an isotopical action on the cover. Combining the planar unwrapping theorem with the boundary control result yields that if a finitely generated $\RFRS$ (or $\RFRp$) group $B$ acts isotopically on a planar surface $S$, then the semidirect product $\pi_1(S)\rtimes B$ is again $\RFRS$ (respectively $\RFRp$). Applying this proposition inductively along the Fadell--Neuwirth splitting proves Theorem~\ref{athm-rfrs} in \S \ref{section:mainthm-proof}. Finally, Corollary~\ref{cor-artin-rfrs} follows from well-known virtual embeddings of the relevant Artin groups into braid groups, together with the fact that being virtually $\mathrm{RFR}p$ passes to subgroups.

\subsection*{Acknowledgments}
Wu is currently a member of LMNS at Fudan University. We thank Ian Agol and Sam Fisher for their helpful communications. We benefited from the assistance of AI tools in our study. In particular,  the key idea  for the proof of Theorem \ref{athm-rfrs} was developed from a candidate proof sketch generated by ChatGPT-5.5 pro.

\section{Partial $\RFRS$ towers} \label{sec:par-RFRS}

In this section, we first recall a way to detect RFRSness using partial $\RFRS$ towers from \cite[\S 3]{HWY26} then generalize it to the $\RFRp$ setting. 

For a group \(G\), define
\[
\Qrad(G)=\ker\bigl(G\longrightarrow H_1(G;\bQ)\bigr).
\]
Thus \(g\in \Qrad(G)\) if and only if the image of \(g\) in \(G^{\ab}\) is torsion.

\begin{defn}
A group \(G\) is called $\RFRS$ if there is a descending sequence
\[
G=G_0\geq G_1\geq G_2\geq \cdots
\]
such that each \(G_{i+1}\) is finite-index normal in \(G_i\),
\[
\Qrad(G_i)\leq G_{i+1},
\]
and
\[
\bigcap_{i\geq 0}G_i=1.
\]
\end{defn}

\begin{defn}
A finite sequence
\[
G=G_0\geq G_1\geq \cdots \geq G_N
\]
is called a \emph{partial $\RFRS$ tower} if, for every \(0\leq i<N\), the subgroup \(G_{i+1}\) is finite-index normal in \(G_i\) and
\[
\Qrad(G_i)\leq G_{i+1}.
\]
The tower detects \(g\in G\) if
\[
g\in G_N
\qquad\text{and}\qquad
[g]\neq 0\in H_1(G_N;\bQ).
\]
\end{defn}

We have the following criterion for detecting $\RFRS$; see \cite[Corollary 3.4]{HWY26}.

\begin{lem}\label{lem:partial-criterion}
Let \(G\) be finitely generated. Then \(G\) is $\RFRS$ if and only if every nontrivial element of \(G\) is detected by a partial $\RFRS$ tower starting at \(G\).
\end{lem}

For any prime $p$, Koberda and Suciu   generalized the definition of $\RFRS$ to $\RFRp$ in \cite{KoberdaSuciu2020}. A $\RFRp$ group is $\RFRS$; see \cite[Lemma 2.1]{KoberdaSuciu2020}.

\begin{defn}

Let $G=G_0$ be a finitely generated group and let $p$ be a prime.
We say that $G$ is \emph{residually finite rationally $p$} (or $\RFRp$)
if there exists a sequence of subgroups $\{G_i\}_{i\geq 0}$ of $G$ such that:

\begin{enumerate}
\item For each $i$, we have $\Qrad(G_i)\leq G_{i+1}$.

\item  $\bigcap_{i\geq 0} G_i=\{1\}$.

\item For each $i$, the quotient group $G_i/G_{i+1}$ is a finite elementary
abelian $p$-group.

\end{enumerate}
\end{defn}

Let $K_0=G$ and define 
\[
K_{i+1}=ker\{K_i \rightarrow H_1(K_i; \mathbb{Z})_{\mathrm{torsionfree}} \rightarrow H_1(K_i; \mathbb{Z})_{\mathrm{torsionfree}} \otimes \bZ/p\bZ \},
\]
where $H_1(K_i;\bZ)_{\mathrm{torsionfree}}$  is  the first homology group modulo the torsion subgroup. The $\RFRp$ radical of $G$ is defined as 

\[
\mathrm{rad}_p(G)= \cap_i K_i.
\]

The group $G$ is $\RFRp$ if and only if the radical $\mathrm{rad}_p(G)$  is trivial. Similarly, one can define a partial $\RFRp$ tower as the following.

\begin{defn}
    Let \(p\) be a prime. A finite  sequence
\[
G=G_0\geq G_1\geq \cdots \geq G_N
\]
is a \emph{partial  $\RFRp$ tower} if, for every \(0\leq i<N\), $
\operatorname{rad}_{\mathbb Q}(G_i)\leq G_{i+1}$, and the quotient \(G_i/G_{i+1}\) is a finite elementary abelian
\(p\)-group.

The tower detects \(g\in G\) if $g\in G_{N-1}$ but \(g\notin G_N\).
\end{defn}

\begin{rem} \label{rem:diff-tower}
    The detection condition for partial $\RFRp$ towers is different from the $\RFRS$ case. For $\RFRS$, a partial tower detects \(g\) if \(g\) survives to the terminal group and has a nonzero image in \(H_1(-;\mathbb Q)\). For $\RFRp$, however, detection means that \(g\) is excluded by an elementary abelian \(p\)-quotient; after truncating the tower, this means \[ g\in G_{N-1}\setminus G_N . \] For example, let us consider $G = \bZ = \langle a\rangle \geq \langle a^p\rangle$. Then $a^{p^2}$ lies in $\langle a^p\rangle$, and it also survives in $H_1(\langle a^p\rangle;\bQ)$. But it does not survive in any elementary $p$-quotient of $\langle a^p\rangle$. Thus $a^{p^2}$ is not detected by the partial $\RFRp$ tower $\langle a\rangle \geq \langle a^p\rangle$.
\end{rem}

\begin{lem}
\label{rfrp tower}
 Let \(G\) be finitely generated and let \(p\) be a prime. Then \(G\)
is  $\RFRp$ if and only if every nontrivial element of \(G\) is detected by a
partial  $\RFRp$ tower starting at \(G\).
\end{lem}

\begin{proof} 
The proof is similar to that of Lemma \ref{lem:partial-criterion}. If \(G\) is $\RFRp$, then every nontrivial element is excluded by some
finite initial segment of an  $\RFRp$ filtration.

Conversely, assume that every nontrivial element of \(G\) is detected by some
partial  $\RFRp$ tower. This property passes to finite-index subgroups. Indeed,
let \(L\leq G\) have finite index and let \(1\neq \ell\in L\). Choose a partial
 $\RFRp$ tower
\[
G=G_0\geq G_1\geq \cdots \geq G_N
\]
with \(\ell\notin G_N\), and put \(L_i=L\cap G_i\). Then
\[
L=L_0\geq L_1\geq \cdots \geq L_N
\]
is a partial  $\RFRp$ tower: normality is immediate, \(L_i/L_{i+1}\) injects
into \(G_i/G_{i+1}\), and if \(x\in \operatorname{rad}_{\mathbb Q}(L_i)\), then
the image of \(x\) in \(H_1(G_i;\mathbb Q)\) is zero, so
\(x\in \operatorname{rad}_{\mathbb Q}(G_i)\leq G_{i+1}\), hence
\(x\in L_{i+1}\). Also \(\ell\notin L_N\).

Now enumerate the nontrivial elements of \(G\) as
\[
g_1,g_2,g_3,\dots .
\]
We build an infinite tower by induction. Suppose a finite partial tower has already been constructed and has a terminal subgroup \(L\). If \(L=\{1\}\), keep the tower constant at \(1\) thereafter. If \(g_j\notin L\), do nothing at the \(j\)-th stage. If \(g_j\in L\), apply the condition to find a partial tower 
$G=H_0 \rhd H_1 \rhd H_2 \rhd \cdots \rhd H_m $
with \(g_j \in H_{m-1} \setminus H_m \). By the previous argument, we have a finite partial tower from \(G\) to a subgroup \(M=L\cap H_m\). Repeat this argument for any $g_i$ to find an infinite tower.  The resulting tower
satisfies the  $\RFRp$ conditions, and every nontrivial element of \(G\) is
eventually excluded in the sequence. Hence, the intersection is trivial. Thus \(G\) is  $\RFRp$.
\end{proof}

\begin{cor}\label{cor:rfrp-tower}
    Let \(G\) be a finitely generated group and let \(p\) be a prime. Then \(G\)
is  $\RFRp$ if and only if every nontrivial element of \(G\) is detected by a
partial  $\RFRp$ tower  $G=G_0\geq G_1\geq \cdots G_n$ such that $[G_i:G_{i+1}] =p$ for $0\leq i\leq n-1$.
\end{cor}
\proof
By Lemma \ref{rfrp tower}, it suffices to prove the ``only if'' direction. For any $g\neq 1$, again by Lemma \ref{rfrp tower}, we have a partial $\RFRp$ tower $G=G_0\geq G_1\geq \cdots G_n$ that detects $g$. Since each $G_{i}/G_{i+1}$ is an elementary abelian $p$-group, it is isomorphic to $(\bZ/p)^{m_i}$ for some $m_i\geq 1$. Let $q_i: G_i \to G_i/G_{i+1}$ be the quotient map, and $G_{i,j}$ be the preimage of the first $(m_i-j)$ copies of $\bZ/p$ in  $(\bZ/p)^{m_i}$ for any $1\leq j\leq m_i$. Now replace $G_i\geq G_{i+1}$ in the partial $\RFRp$ tower by $G_i \geq G_{i,1}\geq \cdots G_{i,m_i} =G_{i+1}$.  Note that $\operatorname{rad}_{\mathbb Q}(G_{i,j}) \leq G_{i,{j+1}}$. In fact, if $x\in \operatorname{rad}_{\mathbb Q}(G_{i,j})$, then for some $k\neq 0$, $x^k$ lies in the commutator subgroup of $G_{i,j}$. Since $G_{i,j}\leq G_i$, $x^k$ also lies in the commutator subgroup of $G_{i}$, hence $x\in \operatorname{rad}_{\mathbb Q}(G_{i})$  so it lies in $G_{i+1}\leq G_{i,j+1}$. Now truncate the refined tower at the first subgroup no containing $g$.
\qed

The following is a topological interpretation of Corollary \ref{cor:rfrp-tower}.
\begin{lem}
\label{rfrp topology}
 Let $X$ be a topological space with an  $\RFRp$ fundamental group for some prime $p$. For any nontrivial element $g\in \pi_1(X)$, there is a tower of $\mathbb{Z}/p$-covers
 \[
        S_r\longrightarrow S_{r-1}\longrightarrow \cdots\longrightarrow S_0=X
\]
such that $g$ lifts throughout the whole tower, and its lift in the final cover $S_r$ is  non-trivial in $H_1(S_r;\bQ)$.
\end{lem}
\proof
Since $\pi_1(X)$ is $\RFRp$, we can find a partial $\RFRp$ tower $\pi_1(X) = G_0\geq G_1\geq \cdots \geq G_r\geq G_{r+1}$, such that $g\in G_{r}$ but $g\not\in G_{r+1}$. Since $\operatorname{rad}_{\mathbb Q}(G_{r}) \leq G_{r+1}$, we have $[g]\neq 0\in  H_1(S_r;\bQ)$. Now take the tower of covers corresponding to partial $\RFRp$ tower $\pi_1(X) = G_0\geq G_1\geq \cdots \geq G_r$, we are done.

\qed

\begin{cor}
\label{curves}
   Let $S$ be a compact surface that is not the real projective plane, the Klein
bottle, nor the non-orientable surface of Euler characteristic $-1$.  For any nontrivial element $g\in \pi_1(S)$, there is a tower of $\mathbb{Z}/p$-covers
 \[
        S_r\longrightarrow S_{r-1}\longrightarrow \cdots\longrightarrow S_0=S
\]
such that $g$ lifts to closed curves throughout the whole tower, and its lift in the final cover $S_r$ is nontrivial in $H_1(S_r;\bQ)$.
\end{cor}

\begin{proof}
    By the previous lemma, it is enough to note that the fundamental group $\pi_1(S)$  embeds as a subgroup into a right-angled Artin group (RAAG) and a RAAG has $\RFRp$ (see \cite[Theorem 3]{MR2077673}, \cite[Proposition 1.1]{KoberdaSuciu2020}), implying that $\pi_1(S)$ has $\RFRp$.
\end{proof}

\section{Planar arc-cyclic covers}\label{sec:cover}

Let us first recall the definition of mapping class group.

\begin{defn}
    Let $S$ be a finite type surface. The \emph{mapping class group} of $S$ is $\Map(S) = \pi_0(\Homeo(S, \partial S))$, the group of isotopy classes of homeomorphisms of $S$ that restrict to the identity on $\partial S$. The pure mapping class group $\PMap(S)$ is the subgroup of $\Map(S)$ that fixes all the punctures pointwise.
\end{defn}

From now on, $S$ will denote a compact connected oriented planar surface with nonempty boundary.  Thus $S$ is a sphere with finitely many open disks  removed.  We assume $\pi_1(S)$ is nontrivial unless explicitly stated; the disk case is vacuous.

A closed curve $\alpha$ in $S$ will always mean a smooth immersed map
\[
        \alpha:S^1\to \operatorname{int}(S)
\]
in general position: all self-intersections are transverse double points, and there are no triple points.  Its number of self-intersections is the number of unordered pairs $\{x,y\}\subset S^1$, $x\ne y$, with $\alpha(x)=\alpha(y)$.  A representative of a free homotopy class is called \emph{minimal} if this number is minimal among all immersed representatives of that free homotopy class. We shall call a closed curve \emph{essential} if it is not null-homotopic. 

Let $q$ be a self-intersection point of $\alpha$.  The two preimages of $q$ divide $S^1$ into two arcs, and their images give two closed curves based at $q$.  We call either of these curves a \emph{self-intersection subloop} of $\alpha$ at $q$.  A self-intersection subloop is called \emph{simple} if its image is an embedded circle. The following two lemmas are elementary. 

\begin{lem}[Innermost essential subloop]\label{lem:essential-subloop}
Let $\alpha$ be a minimal immersed representative of a nontrivial free homotopy class in a compact planar surface $S$.  If $\alpha$ has a self-intersection, then $\alpha$ has a simple self-intersection subloop $\beta$ that is essential in $S$.
\end{lem}

\begin{proof}
Among all self-intersection subloops of $\alpha$, choose one whose number of self-intersections is minimal.  If this number were positive, then a self-intersection of that subloop would cut off a smaller self-intersection subloop, contradicting the choice.  Hence the chosen subloop $\beta$ is simple.

It remains to show that $\beta$ is essential.  Suppose, for contradiction, that $\beta$ is null-homotopic.  Since $S$ is a surface and $\beta$ is simple, $\beta$ bounds an embedded disk $D\subset S$.  Choose such a null-homotopic simple self-intersection subloop with the number of intersections of $\alpha$ inside $D$ minimal.  If the interior of $D$ met $\alpha$, then an innermost component of the intersection pattern inside $D$ would produce a smaller null-homotopic simple self-intersection subloop, contradicting the minimality of $D$. Note that here, the innermost component of the intersection pattern inside $D$ would not be a bigon since $\alpha$ is in minimal position. Therefore the interior of $D$ is disjoint from $\alpha$.

Thus $\beta$ is a monogon bounded by a subarc of $\alpha$.  Sliding this subarc across $D$ removes the corresponding self-intersection and creates no new self-intersections.  This contradicts the minimality of $\alpha$.  Hence $\beta$ is essential.
\end{proof}

\begin{lem}[Essential simple curves in planar surfaces]\label{lem:essential-separates}
Let $\beta\subset S$ be an essential simple closed curve.  Then $\beta$ separates the boundary components of $S$ into two nonempty proper collections.  In particular, there is a properly embedded arc $a\subset S$ joining two boundary components and meeting $\beta$ transversely in exactly one point.
\end{lem}

\begin{proof}
Since $S$ is planar, every simple closed curve separates $S$.  If one complementary component of $S\setminus \beta$ contained no boundary component of $S$, then that component would be a disk, and $\beta$ would be null-homotopic.  Since $\beta$ is essential, both complementary components contain boundary components.  Choose one boundary component on each side and join them to a point in $\beta$ by two embedded arcs in each component. The concatenation of the two arcs  gives the required arc.
\end{proof}

\subsection{Cyclic covers dual to arcs}\label{sec:cyc-dul-cov}

Let $a\subset S$ be a properly embedded oriented arc whose endpoints lie on two distinct boundary components.  The algebraic intersection number (for details, see \cite[p. 28 and 165]{FaMa11}) with $a$ defines a homomorphism
\[
        \lambda_a:H_1(S;\mathbb{Z})\longrightarrow \mathbb{Z},
\]
and hence also a homomorphism, denoted by the same symbol,
\[
        \lambda_a:\pi_1(S)\longrightarrow \mathbb{Z}.
\]
For an integer $m\ge 2$, let
\[
        S_{a,m}\longrightarrow S
\]
be the connected cyclic $m$-fold cover corresponding to
\[
        \ker\bigl(\pi_1(S)\xrightarrow{\lambda_a}\mathbb{Z}\to \mathbb{Z}/m\mathbb{Z}\bigr).
\]
We call this an \emph{arc-cyclic $m$-cover}; see Figure \ref{fig:sample} for a picture.

\begin{figure}[ht]
    \centering
\includegraphics[width=0.4\textwidth]{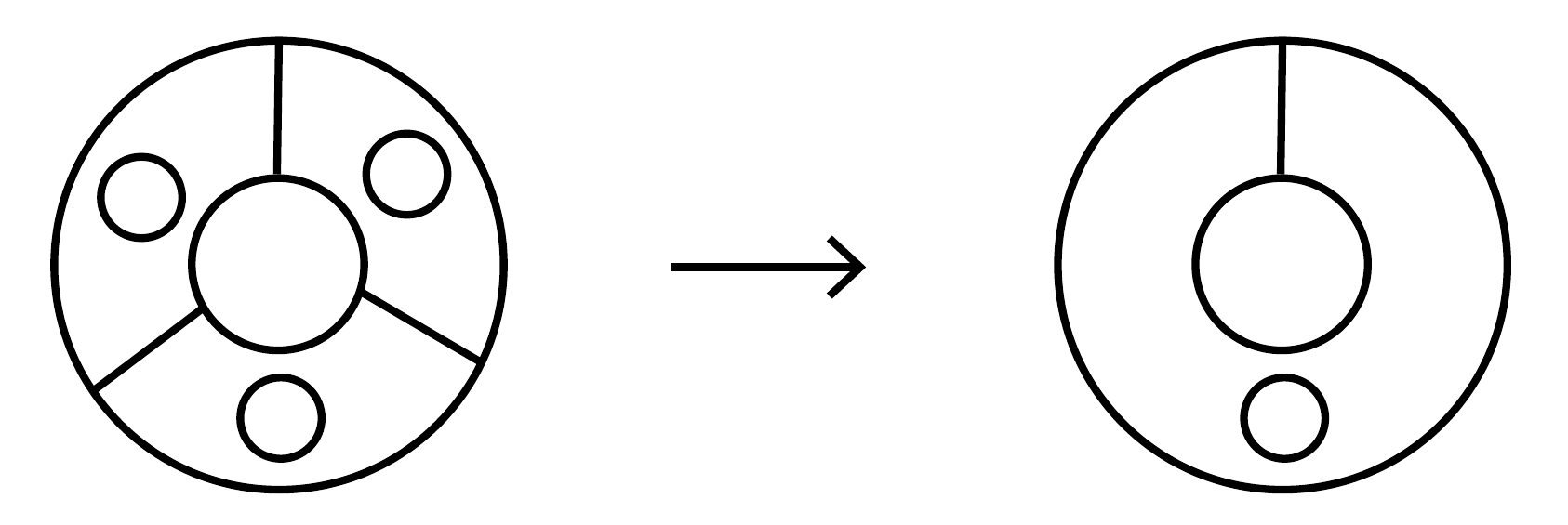}
    \caption{Arc-cyclic $3$-cover}
    \label{fig:sample}
\end{figure}

\begin{lem}[Arc-cyclic covers are planar]\label{lem:cyclic-planar}
Every arc-cyclic cover $S_{a,m}\to S$ is planar.
\end{lem}

\begin{proof}
Let $b$ be the number of boundary components of $S$.  Let $C_1$ and $C_2$ be the two boundary components containing the endpoints of $a$.  With suitable boundary orientations, we can assume that
\[
        \lambda_a([C_1])=1,
        \qquad
        \lambda_a([C_2])=-1,
\]
and $\lambda_a([C])=0$ for every other boundary component $C$.

In the cyclic $m$-fold cover, a boundary component $C$ with $\lambda_a([C])=k$ has $\gcd(m,k)$ preimage components.  Therefore $C_1$ and $C_2$ each lift to one boundary component, while each of the remaining $b-2$ boundary components lifts to $m$ boundary components.  Hence
\[
        |\partial S_{a,m}|=2+m(b-2).
\]
Also
\[
        \chi(S_{a,m})=m\chi(S)=m(2-b).
\]
If $g$ is the genus of $S_{a,m}$, then
\[
        m(2-b)=\chi(S_{a,m})
        =2-2g-|\partial S_{a,m}|
        =2-2g-\bigl(2+m(b-2)\bigr)
        =m(2-b)-2g.
\]
Thus $g=0$.  Hence $S_{a,m}$ is planar.
\end{proof}

\subsection{The unwrapping criterion} We develop techniques here for resolving intersections by passing to covers.

Let $\pi:\tilde{S} \to S$ be a covering, let $\alpha: S^1 \to S $ be a closed loop and denote its lift by $\tilde{\alpha}$. Let    $q$ be a self-intersection of $\alpha$, and  $x,y\in S^1$ be the two preimages of $q$. We shall call the two local branches of $\alpha$ at $x,y$  the \emph{two branches of  $\alpha$ at $q$}. Choose one of the two self-intersection subloops $\beta$ obtained by traversing $\alpha$ from $x$ to $y$. Let
\[
\widetilde q=\widetilde\alpha(x).
\]
The lift of $\beta$ beginning at $\widetilde q$ ends at $\widetilde\alpha(y)$, which is another lift of $q$. Therefore the images of $x$ and $y$ under $\alpha$ lift to the same point of $\widetilde S$ (i.e., $\tilde\alpha(x) = \tilde\alpha(y)$) if and only if the lift of $\beta$ starting at $\widetilde q$ is a loop. Equivalently,
\[
\widetilde\alpha(x)=\widetilde\alpha(y),
\]
i.e., the lifted path $\widetilde\beta$ returns to $\widetilde q$. Thus the self-intersection $q$ survives in the lifted curve $\widetilde\alpha$ precisely when $\widetilde\beta$ is a loop.

\begin{lem}[No new intersections under lifting]\label{lem:no-new}
Let $\pi:\widetilde S\to S$ be a covering map, and let $\alpha$ lift to a closed curve $\widetilde\alpha$.  Then every self-intersection of $\widetilde\alpha$ maps to a self-intersection of $\alpha$. Consequently, lifting creates no new self-intersections.  Moreover, if at some self-intersection $q$ of $\alpha$ the two branches lift to two distinct points of $\widetilde S$, then the number of self-intersections of $\widetilde\alpha$ is strictly smaller than that of $\alpha$.
\end{lem}

\begin{proof}
If $x,y\in S^1$, $x\ne y$, and $\widetilde\alpha(x)=\widetilde\alpha(y)$, then applying $\pi$ gives $\alpha(x)=\alpha(y)$.  Thus every self-intersection pair of $\widetilde\alpha$ projects to a self-intersection pair of $\alpha$.  No new self-intersection pair can appear in the lift.

If the two branches at a self-intersection $q$ lift to distinct points, then the corresponding self-intersection pair of $\alpha$ does not give a self-intersection pair of $\widetilde\alpha$.  Since no new pairs appear, the total number strictly decreases.
\end{proof}

\begin{thm}[Planar finite-cover unwrapping criterion]\label{thm:unwrapping-criterion}
Let $S$ be a compact connected oriented planar surface with nonempty boundary.  Let $\alpha$ be a minimal immersed closed curve in $S$, and let $\beta$ be a simple essential self-intersection subloop of $\alpha$.  Suppose there exist a properly embedded arc $a\subset S$ joining two distinct boundary components and an integer $m\ge 2$ such that
\[
        \lambda_a([\alpha])\equiv 0 \pmod m,
        \qquad
        \lambda_a([\beta])\not\equiv 0 \pmod m.
\]
Then $\alpha$ has a closed lift to the planar finite cover $S_{a,m}\to S$, and this lift has strictly fewer self-intersections than $\alpha$.
\end{thm}

\begin{proof}
The congruence $\lambda_a([\alpha])\equiv 0\pmod m$ says exactly that the element represented by $\alpha$ lies in the subgroup defining the cyclic cover $S_{a,m}\to S$.  Hence $\alpha$ lifts to a closed curve $\widetilde\alpha$ in $S_{a,m}$.

At the chosen self-intersection, say $q$, the subloop $\beta$ changes the sheet by the deck transformation corresponding to $\lambda_a([\beta])\in \mathbb{Z}/m\mathbb{Z}$. Suppose that $x,y\in S^1$ map to $q$ under $\alpha$.  Since  $\lambda_a([\beta])$ is nonzero, the images of $x,y$ under $\alpha$  lift to two distinct points of $S_{a,m}$.  Lemma~\ref{lem:no-new} therefore shows that the lift has strictly fewer self-intersections.  Finally, Lemma~\ref{lem:cyclic-planar} shows that $S_{a,m}$ is planar.
\end{proof}

\begin{cor}[Null-homologous one-step unwrapping]\label{cor:null-one-step}
Let $1\ne f\in \pi_1(S)$, and let $\alpha$ be a minimal immersed representative of its free homotopy class.  Suppose that
\[
        [f]=0\in H_1(S;\mathbb{Z})
\]
and  $\alpha$ has a self-intersection.  For any integer $m>1$, there is an arc-cyclic $m$-cover $S'\to S$ such that $f$ lifts closed to $S'$ and the lifted curve has strictly fewer self-intersections.
\end{cor}

\begin{proof}
By Lemma~\ref{lem:essential-subloop}, $\alpha$ has a simple essential self-intersection subloop $\beta$.  By Lemma~\ref{lem:essential-separates}, choose a properly embedded arc $a$ joining boundary components on opposite sides of $\beta$ and crossing $\beta$ exactly once.  Then
\[
        \lambda_a([\beta])=\pm 1.
\]
Since $[f]=0\in H_1(S;\mathbb{Z})$, we have $\lambda_a([\alpha])=0$.  Applying Theorem~\ref{thm:unwrapping-criterion} with the given $m\ge 2$ gives the desired   arc-cyclic $m$-cover.
\end{proof}

\subsection{A finite planar tower}

The following statement is the key result needed in the $\RFRS$ argument for pure braid groups.

\begin{thm}[Planar tower detection]\label{thm:tower}
Let $S$ be a compact connected oriented planar surface with nonempty boundary, and let $1\ne f\in\pi_1(S)$.  For any integer $m>1$,  there is a finite tower of arc-cyclic $m$-covers
\[
        S_r\longrightarrow S_{r-1}\longrightarrow \cdots\longrightarrow S_0=S
\]
such that $f$ lifts closed through the whole tower, and the final lift $f_r\in\pi_1(S_r)$ has a nonzero image in $H_1(S_r;\mathbb{Z})$.
\end{thm}

\begin{proof}
Let $f_i\in\pi_1(S_i)$ denote the current lift of $f$.  If
\(
        [f_i]\ne 0\in H_1(S_i;\mathbb{Z}),
\)
we stop.  Otherwise,  we have
\(
        [f_i]=0\in H_1(S_i;\mathbb{Z}).
\)
Choose a minimal immersed representative $\alpha_i$ of the free homotopy class of $f_i$.

If $\alpha_i$ has no transverse self-intersections, then its image is an essential simple closed curve.  In a planar surface, every essential simple closed curve separates the boundary components into two nonempty collections.  Its homology class is the sum, up to sign, of the boundary classes on one side.  Since the only relation among boundary classes is that the total sum is zero, the sum over a nonempty proper subcollection is nonzero in $H_1(S_i;\mathbb{Z})$.  Thus $[f_i]\ne 0\in H_1(S_i;\bZ)$, contradicting the assumption that $[f_i]=0$.  Hence $\alpha_i$ has a self-intersection.

By Corollary~\ref{cor:null-one-step}, there is an arc-cyclic $m$-cover
\[
        S_{i+1}\to S_i
\]
such that $f_i$ lifts closed to $f_{i+1}\in\pi_1(S_{i+1})$ and the minimal self-intersection number of the lifted free homotopy class is strictly smaller.  By Lemma~\ref{lem:cyclic-planar}, $S_{i+1}$ is planar.

At every nonterminal stage the minimal self-intersection number strictly decreases, and it is a nonnegative integer.  Therefore the process terminates after finitely many steps.  At the terminal stage, the lift has a nonzero homology class.
\end{proof}

To prove that the pure braid groups are $\RFRp$, we need the following   strengthening (which is Theorem \ref{athm:maintower}) of Theorem~\ref{thm:tower}.

\begin{thm}[Planar tower detection mod \(p\)] \label{thm:tower-p}
Let \(S\) be a compact connected oriented planar surface with nonempty
boundary, and let \(1\ne f\in\pi_1(S)\). For any prime \(p\), there is
a finite tower of arc-cyclic \(p\)-covers
\[
        S_r\longrightarrow S_{r-1}\longrightarrow \cdots\longrightarrow S_0=S
\]
such that \(f\) lifts closed through the whole tower,  $S_i$ is planar  for all $i$, and the final lift
\(f_r\in\pi_1(S_r)\) has nonzero image in
\(H_1(S_r;\mathbb Z/p)\).
\end{thm}

\begin{proof}

Apply Theorem~\ref{thm:tower} with \(m=p\). We obtain a finite tower of
arc-cyclic \(p\)-covers
\[
S_s\longrightarrow S_{s-1}\longrightarrow \cdots \longrightarrow S_0=S
\]
such that \(f\) lifts closed through the tower and the final lift
\(f_s\in \pi_1(S_s)\) satisfies
\[
0\ne [f_s]\in H_1(S_s;\mathbb Z).
\]
If the image of \([f_s]\) in \(H_1(S_s;\mathbb Z/p)\) is nonzero, then
we are done. Thus suppose that \([f_s]\in pH_1(S_s;\mathbb Z)\).

We claim that, whenever \(T\) is a compact connected oriented planar
surface with nonempty boundary and \(g\in\pi_1(T)\) satisfies
\[
0\ne [g]\in H_1(T;\mathbb Z),
\]
one can pass to an arc-cyclic \(p\)-cover to strictly decrease the
\(p\)-divisibility of \([g]\), provided \([g]\) is divisible by \(p\).
Indeed, write
\[
[g]=p^k u
\]
where \(k\ge 1\) is maximal and \(u\notin pH_1(T;\mathbb Z)\). Since
\(T\) is planar, the homomorphisms \(\lambda_a\) obtained by algebraic
intersection with properly embedded arcs \(a\) joining distinct boundary
components span \(H^1(T;\mathbb Z/p)\). Hence we may choose such an arc
\(a\) with
\[
\lambda_a(u)\not\equiv 0 \pmod p.
\]
Let
\[
q:T'=T_{a,p}\longrightarrow T
\]
be the corresponding arc-cyclic \(p\)-cover. Since
\[
\lambda_a([g])=\lambda_a(p^k u)\equiv 0 \pmod p,
\]
the element \(g\) lifts closed to some \(g'\in\pi_1(T')\). Put
\(x'=[g']\in H_1(T';\mathbb Z)\). Then
\[
q_*(x')=[g]=p^k u,
\]
implying \(x'\ne 0\).

We show that \(x'\notin p^kH_1(T';\mathbb Z)\). If \(x'=p^k y\), then
\[
p^k q_*(y)=q_*(x')=p^k u.
\]
Since \(H_1(T;\mathbb Z)\) is torsion-free, it follows that
\(q_*(y)=u\). But every class in the image of \(q_*:H_1(T';\mathbb Z)\to
H_1(T;\mathbb Z)\) lies in the kernel of
\[
H_1(T;\mathbb Z)\xrightarrow{\lambda_a}\mathbb Z\longrightarrow
\mathbb Z/p,
\]
because \(T'\to T\) is the cover corresponding to this mod-\(p\)
character. Thus \(\lambda_a(u)\equiv 0\pmod p\), contradicting the choice
of \(a\). Therefore, the maximal \(p\)-divisibility of \(x'\) is strictly
smaller than that of \([g]\).

Starting with \(T=S_s\) and \(g=f_s\), repeat the preceding step as long
as the current integral homology class is divisible by \(p\). The
maximal \(p\)-divisibility is a nonnegative integer and strictly
decreases at each step, which implies the process terminates after finitely many
steps. We therefore obtain a further finite tower of arc-cyclic
\(p\)-covers
\[
S_r\longrightarrow \cdots \longrightarrow S_s
\]
such that the final lift \(f_r\in \pi_1(S_r)\) satisfies
\[
[f_r]\notin pH_1(S_r;\mathbb Z).
\]
Equivalently, the image of \(f_r\) in \(H_1(S_r;\mathbb Z/p)\) is
nonzero.

Concatenating this tower with the tower supplied by Theorem~\ref{thm:tower} gives
the desired finite tower of \(p\)-covers. Since all covers used are
arc-cyclic covers of planar surfaces, the intermediate surfaces remain
planar.
\end{proof}

\section{Boundary control}\label{sec:b-contr}

A group \(B\) acts \emph {isotopically} on a compact planar surface \(S\) if it acts on \(S\) by orientation-preserving mapping classes fixing the basepoint and  the boundaries of \(S\) pointwise. In other words, up to isotopy, each element in $B$ is viewed as a homeomorphism of $S$ via a group homomorphism $B\rightarrow \Map(S)$. Note that we do not require the homomorphism to be injective here. An isotopical action induces a group action on the fundamental group $\pi_1(S)$ with the basepoint in the boundary of $S$. Note that the induced action  on \(H_1(S;\bZ)\) is trivial, because \(H_1(S;\bZ)\) is generated by the boundary classes, with the single relation that their total sum is zero.
\begin{thm}[Boundary lifting]\label{lem:boundary-correction}
Let \(S\) be a compact planar surface with fundamental group \(K=\pi_1(S)\), and let \(B\) act isotopically on \(S\). Let \(S'\to S\) be an arc-cyclic $m$-cover with corresponding subgroup
\(
K'\leq K.
\)
Then \(K'\) is \(B\)-invariant, and there is a finite-index normal subgroup
\(
B^+\leq B
\)
such that:
\begin{enumerate}[label=\textup{(\roman*)}]
\item The canonical lifts of elements of \(B^+\) to \(S'\) preserve every boundary component of \(S'\) pointwise; in other words, there is an induced isotopical action of  $B^+$  on $S'$.
\item
\[
K\rtimes B\geq K'\rtimes B\geq K'\rtimes B^+
\]
is a partial $\RFRS$ tower.

\item If \(m=p\) is a prime, then the two-step tower\[
K\rtimes B\geq K'\rtimes B\geq K'\rtimes B^+
\]
is a partial $\RFRp$ tower.

\end{enumerate}
\end{thm}

\begin{proof}

Let the arc-cyclic \(m\)-cover be defined by the mod-\(m\) reduction of
an integral character
\[
\lambda\colon K\to \bZ.
\]
Since \(B\) acts trivially on \(H_1(S;\bZ)\), it preserves \(\lambda\).
Hence it preserves
\[
K'=\ker\bigl(K\overset{\lambda}{\longrightarrow}\bZ
\longrightarrow \bZ/m\bZ\bigr),
\]
and therefore \(K'\rtimes B\) is a subgroup of \(K\rtimes B\).

We claim that \(K\rtimes B\geq K'\rtimes B\) is a partial $\RFRS$ tower.
Indeed, \(\lambda\) extends to a homomorphism
\[
\widehat\lambda\colon K\rtimes B\to \bZ
\]
by declaring \(\widehat\lambda\) to be zero on \(B\). This is well-defined
because \(B\) preserves \(\lambda\). The quotient map
\[
K\rtimes B\longrightarrow (K\rtimes B)/(K'\rtimes B)\cong \bZ/m\bZ
\]
is the mod-\(m\) reduction of \(\widehat\lambda\). Therefore every element
of \(\Qrad(K\rtimes B)\) lies in \(K'\rtimes B\).

We proceed to define \(B^+\). Let
\[
S_\lambda\to S
\]
be the infinite cyclic cover corresponding to \(\ker(\lambda)\), and let
\(t\) be a chosen generator of its deck group. Choose a lift
\(\widetilde x_0\) of the basepoint \(x_0\) in the boundary to \(S_\lambda\). Since every
\(b\in B\) preserves \(\lambda\), the homeomorphism \(f_b\) representing
\(b\) has a unique canonical lift \(\widetilde f_b\) to \(S_\lambda\)
fixing \(\widetilde x_0\) \cite[Proposition 1.33 and 1.34]{Hat02}. This lift
commutes with \(t\).

For each boundary component \(C\subset \partial S\), choose a point
\(z_C\in C\). If the basepoint \(x_0\) lies on \(C\), we choose
\(z_C=x_0\). Choose a lift \(\widetilde z_C\in S_\lambda\) of \(z_C\),
taking \(\widetilde z_C=\widetilde x_0\) when \(z_C=x_0\). Since \(f_b\)
fixes \(z_C\), there is a unique integer \(e_C(f_b)\) such that
\[
\widetilde f_b(\widetilde z_C)=t^{e_C(f_b)}\widetilde z_C.
\]

Note that if \(g_b\in \Homeo(S,\partial S)\) is another homeomorphism
representing \(b\), then \(g_b\) is isotopic to \(f_b\) relative to
\(\partial S\) and the basepoint. Lifting this isotopy to \(S_\lambda\),
starting at the canonical lift of \(f_b\), gives an isotopy to the
canonical lift of \(g_b\). Since the isotopy fixes \(z_C\) downstairs,
the lifted path starting at \(\widetilde f_b(\widetilde z_C)\) stays in
the discrete fiber over \(z_C\), and hence is constant. Therefore
\[
e_C(f_b)=e_C(g_b).
\]
We denote \(e_C(b)=e_C(f_b)\).

Choosing representatives \(f_{b_1}\) and \(f_{b_2}\), the composition
\(\widetilde f_{b_1}\widetilde f_{b_2}\) is the canonical lift of the
representative \(f_{b_1}f_{b_2}\) of \(b_1b_2\). Since canonical lifts
commute with \(t\), we have
\[
\begin{aligned}
\widetilde f_{b_1b_2}(\widetilde z_C)
&=\widetilde f_{b_1}\widetilde f_{b_2}(\widetilde z_C) \\
&=\widetilde f_{b_1}\bigl(t^{e_C(b_2)}\widetilde z_C\bigr) \\
&=t^{e_C(b_2)}\widetilde f_{b_1}(\widetilde z_C) \\
&=t^{e_C(b_2)}t^{e_C(b_1)}\widetilde z_C \\
&=t^{e_C(b_1)+e_C(b_2)}\widetilde z_C .
\end{aligned}
\]
Thus
\[
e_C(b_1b_2)=e_C(b_1)+e_C(b_2).
\]
Hence
\[
e_C\colon B\to \bZ
\]
is a homomorphism.

The cover \(S'\) is the quotient of \(S_\lambda\) by the subgroup
generated by \(t^m\). Let
\[
\overline f_b\colon S'\to S'
\]
be the lift induced by \(\widetilde f_b\). Define
\[
B^+=\bigcap_{C\subset \partial S}
\ker\bigl(B\overset{e_C}{\longrightarrow}\bZ
\longrightarrow \bZ/m\bZ\bigr),
\]
where \(C\) ranges over all boundary components of \(S\). Then \(B^+\)
is finite-index normal in \(B\).

We show that the canonical lifts of elements of \(B^+\) fix every
boundary component of \(S'\) pointwise. Let \(b\in B^+\). Then
\[
e_C(b)\equiv 0\pmod m
\]
for every boundary component \(C\subset \partial S\). Let \(D\) be a
boundary component of \(S'\) lying over \(C\). Then \(D\) contains a point
of the form \([t^j\widetilde z_C]\) for some \(j\), where brackets denote
the image in the quotient \(S_\lambda/\langle t^m\rangle=S'\). Since
\(\widetilde f_b\) commutes with \(t\), we have
\[
\begin{aligned}
\overline f_b([t^j\widetilde z_C])
&=[t^j\widetilde f_b(\widetilde z_C)] \\
&=[t^{j+e_C(b)}\widetilde z_C] \\
&=[t^j\widetilde z_C],
\end{aligned}
\]
where the last equality uses \(e_C(b)\equiv 0\pmod m\). Hence
\(\overline f_b\) fixes a point of \(D\), and therefore sends \(D\) to
itself.

Since \(f_b\) fixes \(C\) pointwise, the restriction
\[
\overline f_b|_D\colon D\to D
\]
is a lift of the identity map on \(C\). A lift of the identity map of a
circle which fixes one point is the identity. Therefore \(\overline f_b\)
fixes \(D\) pointwise. Since \(D\) was arbitrary, every boundary component
of \(S'\) is fixed pointwise. In particular, we have an isotopical action
of \(B^+\) on \(S'\). This proves (i).

Since each \(e_C\) is an integral homomorphism, it vanishes on
\(\Qrad(B)\). Hence
\[
\Qrad(B)\leq B^+.
\]
Now consider the projection
\[
K'\rtimes B\longrightarrow B.
\]
If \(x\in \Qrad(K'\rtimes B)\), then its image in \(B\) lies in
\(\Qrad(B)\), and hence lies in \(B^+\). Therefore
\[
\Qrad(K'\rtimes B)\leq K'\rtimes B^+.
\]
This proves that the second step is also $\RFRS$-legal. The normality and
finite-index conditions are immediate from the definitions of \(K'\) and
\(B^+\).

When \(m=p\), we have
\[
(K\rtimes B)/(K'\rtimes B)\cong \bZ/p\bZ,
\]
and
\[
(K'\rtimes B)/(K'\rtimes B^+)\cong B/B^+.
\]
For each boundary component \(C\subset \partial S\), let
\[
B_C=\ker\bigl(B\overset{e_C}{\longrightarrow}\bZ
\longrightarrow \bZ/p\bZ\bigr).
\]
Since \(B^+\) is the intersection of all \(B_C\), we have an injection
\[
B/B^+ \longrightarrow \prod_C B/B_C.
\]
Each \(B/B_C\) is a subgroup of \(\bZ/p\bZ\). Hence \(B/B^+\) is an
elementary abelian \(p\)-group. The required containment of
\(\Qrad\) is exactly the one proved above.
\end{proof}

\section{The semidirect product step}

In this section, we prepare the proof of Theorem \ref{athm-rfrs} by establishing the following.

\begin{prop}\label{prop:semidirect}
Let \(S\) be a compact connected planar surface with nonempty boundary, and fundamental group
\(
K= \pi_1(S).
\)
Let \(B\) be a finitely generated $\RFRS$ (resp. $\RFRp$ for some prime $p$) group acting isotopically on \(S\). Then
\(
K\rtimes B
\)
is $\RFRS$ (resp. $\RFRp$).
\end{prop}

\begin{proof}
 By Lemma~\ref{lem:partial-criterion} and Lemma \ref{rfrp tower}, it suffices to show that every nontrivial element of \(K\rtimes B\) is detected by a partial $\RFRS$  (resp. $\RFRp$) tower. Let us deal with the $\RFRS$ case first and explain how to prove the $\RFRp$ case at the end. The key difference is that, instead of using Theorem \ref{thm:tower}, one has to use Theorem \ref{thm:tower-p} in the $\RFRp$ case due to the difference between the detection conditions for partial towers; see Remark \ref{rem:diff-tower}.

Let
\[
1\neq g=(f,b)\in K\rtimes B,
\qquad
f\in K,
\quad
b\in B.
\]
First suppose \(b\neq 1\). Since \(B\) is $\RFRS$, Lemma~\ref{lem:partial-criterion} gives a partial $\RFRS$ tower
\[
B=B_0\geq B_1\geq \cdots \geq B_N
\]
which detects \(b\). Define
\[
G_i=K\rtimes B_i.
\]
Then
\[
G=G_0\geq G_1\geq \cdots \geq G_N
\]
is a partial $\RFRS$ tower. Indeed, if \(x\in \Qrad(G_i)\), then the image of \(x\) in \(B_i\) lies in \(\Qrad(B_i)\), hence lies in \(B_{i+1}\). Thus
\[
\Qrad(G_i)\leq K\rtimes B_{i+1}=G_{i+1}.
\]
At the terminal stage, the projection \(G_N\to B_N\) sends \(g\) to \(b\), whose image in \(H_1(B_N;\bQ)\) is nonzero. Therefore \(g\) has a nonzero image in \(H_1(G_N;\bQ)\).

Now suppose \(b=1\), so \(g=f\in K\) and $m$ is any fixed integer greater than $1$. Applying Theorem~\ref{thm:tower},  we obtain arc-cyclic $m$-covers
\[
S_r\longrightarrow S_{r-1}\longrightarrow \cdots \longrightarrow S_0=S
\]
such that \(f\) lifts to closed curves throughout the tower and the final lift \(f_r\in \pi_1(S_r)\) has a nonzero image in \(H_1(S_r;\bZ)\).

Put \(K_i=\pi_1(S_i)\) and \(B_0=B\). We realize the topological tower as a partial $\RFRS$ tower of groups $K =K_0\geq K_1 \geq \cdots\geq K_r$. Suppose inductively that \(B_i\) acts isotopically on \(S_i\). Applying Theorem~\ref{lem:boundary-correction} to the arc-cyclic cover \(S_{i+1}\to S_i\), we obtain a finite-index normal subgroup \(B_{i+1}\leq B_i\) such that
\[
K_i\rtimes B_i
\geq
K_{i+1}\rtimes B_i
\geq
K_{i+1}\rtimes B_{i+1}
\]
is a partial $\RFRS$ tower, and \(B_{i+1}\) acts isotopically on \(S_{i+1}\). Concatenating these two-step towers gives a partial $\RFRS$ tower from \(K\rtimes B\) to
\[
K_r\rtimes B_r.
\]

At the final stage, \(B_r\) preserves every boundary component of the planar surface \(S_r\). Hence \(B_r\) acts trivially on \(H_1(K_r;\bZ)\). Therefore
\[
H_1(K_r\rtimes B_r;\bZ)
\cong
H_1(K_r;\bZ)\oplus H_1(B_r;\bZ).
\]
The element \(g=f\) maps to \(f_r\), which is nonzero in \(H_1(K_r;\bQ)\). Thus \(g\) is detected by the constructed partial tower. By Lemma~\ref{lem:partial-criterion}, \(K\rtimes B\) is $\RFRS$.

We briefly indicate the $\RFRp$ case.  Fix a prime \(p\).  If
\(b\neq 1\), choose a partial $\RFRp$ tower
\[
B=B_0\geq B_1\geq \cdots \geq B_N
\]
detecting \(b\), so \(b\in B_{N-1}\setminus B_N\).  Then
\(G_i=K\rtimes B_i\) gives a partial $\RFRp$ tower detecting
\(g=(f,b)\), by the same projection argument as above.

Now suppose \(b=1\), so \(g=f\in K\).  Instead of Theorem~\ref{thm:tower}, apply
Theorem~\ref{thm:tower-p} to obtain a tower of arc-cyclic \(p\)-covers with final lift
\(f_r\in K_r=\pi_1(S_r)\) satisfying
\[
[f_r]\neq 0\in H_1(S_r;\bZ/p).
\]
Applying Theorem~\ref{lem:boundary-correction} with \(m=p\) at each stage gives, exactly as above, 
a partial $\RFRp$ tower from \(K\rtimes B\) to
\(
H=K_r\rtimes B_r .
\)
Since \(B_r\) acts isotopically on \(S_r\), it acts trivially on
\(H_1(K_r;\bZ)\). Hence the image of \(f\) is nonzero in
\[
H_1(H;\bZ)_{\mathrm{tf}}\otimes \bZ/p .
\]
Let
\[
H^+=
\ker\left(
H\longrightarrow H_1(H;\bZ)_{\mathrm{tf}}\otimes \bZ/p
\right).
\]
Note that \(H/H^+\) is finitely generated elementary abelian \(p\), \(\Qrad(H)\leq H^+\), and
\(f\notin H^+\).  Appending the final step \(H\geq H^+\) gives a partial
 $\RFRp$ tower detecting \(f\).  Thus \(K\rtimes B\) is $\RFRp$.

\end{proof}

\section{Pure braid groups}\label{section:mainthm-proof}

Let $S_n$ be a  disk  with $n$ open disks removed; in particular, it is a connected, oriented, compact, planar surface with nonempty boundary. Let us denote the boundary coming from the disk by $\partial_0$, the remaining ones by $\partial_1,\cdots, \partial_n$, and pick a basepoint $x_0$ in $\partial_0$. Let $D_n$ be the surface obtained from $S_n$ by attaching $n$ punctured disks to all the boundaries that are not $\partial_0$, so $D_n$ is a disk with $n$ punctures. Then
\[
\pi_1(D_n,x_0)\cong \pi_1(S_n,x_0)\cong F_n,
\]
a free group of rank \(n\).

The Fadell--Neuwirth fibration obtained by forgetting the last puncture of $D_n$ gives a split exact sequence
\begin{equation}\label{eq:FN}
1\longrightarrow F_{n-1}\longrightarrow \PMap(D_n) \cong \PB_n
\longrightarrow \PMap(D_{n-1}) \cong\PB_{n-1}\longrightarrow 1.
\end{equation}
Thus
\[
\PB_n\cong F_{n-1}\rtimes \PB_{n-1}.
\]
The action of \(\PB_{n-1}\) on \(F_{n-1}\) is the usual monodromy action of pure mapping classes of $D_{n-1}$. Recall that $\Map(S_n)$ is just the pure ribbon braid group; see for example \cite[Lemma 3.21]{SkiWu25} for a proof. So we  have an  isomorphism $\Map(S_n) \cong \PMap(D_n) \times \bZ^n \cong PB_n \times \bZ^n$, where $\bZ^n$ is generated by Dehn twists along boundaries that are not $\partial_0$. Hence $PB_n$ can be viewed as a subgroup of $\Map(S_n)$, so it acts isotopically on $S_n$. Note that this further induces an isotopical action on $D_n$, hence an action on $\pi_1(D_n,x_0) \cong F_n$. By choosing the embedding properly, this action also coincides with the  usual monodromy action of pure mapping classes of $D_{n}$.

\begin{proof}[Proof of Theorem \ref{athm-rfrs}]
We argue by induction on \(n\). The group \(\PB_1\) is trivial, and
\[
\PB_2\cong \bZ,
\]
so the result is immediate for \(n\leq 2\).

Assume \(n\geq 3\), and assume that \(\PB_{n-1}\) is   $\RFRS$ (resp. $\RFRp$). Now in  the Fadell--Neuwirth splitting \eqref{eq:FN},
\[
\PB_n\cong F_{n-1}\rtimes \PB_{n-1}.
\]
We can view $F_{n-1}$ as $\pi_1(S_{n-1})$, where \(S_{n-1}\) is obtained from $S_n$ by attaching a disk to the boundary $\partial_n$, and \(\PB_{n-1} \leq PB_n\) acts isotopically on \(S_{n-1}\). Proposition~\ref{prop:semidirect} therefore applies and shows that \(\PB_n\) is   $\RFRS$ (resp. $\RFRp$). This completes the induction.
\end{proof}

\begin{proof}[Proof of Corollary \ref{cor-artin-rfrs}]
Let $p$ be any prime. The braid groups contain the pure braid groups as a finite-index subgroup, so they are virtually $\RFRp$. The Artin group of type $B_n=C_n$ embeds into the braid groups (in fact Artin groups of type $A_{2n}$) \cite[Table 1]{Cri99}, so they are virtually $\RFRp$ since being $\RFRp$ passes to subgroups \cite[Theorem 4.3]{KoberdaSuciu2020}. The Artin groups of type $I_2(n)$ are well known to be virtually special, see for example \cite[Lemma 4.3]{HKP16}. In \cite[p. 327]{ChCri05}, it was proved that $A_{B_n} \cong A_{\tilde{A}_{n-1}} \rtimes \bZ$. Therefore, Artin groups of $\tilde{A}_n$ embed into Artin groups of type $B_{n+1}=C_{n+1}$, and they are also virtually $\RFRp$. It was also proved in  \cite[\S 2]{ChCri05} that Artin groups of type $\tilde{C}_n$  embed in the mapping class group $\Map(\Sigma_{n+3})$, where $\Sigma_{n+3}$ is the $2$-sphere with $(n+3)$ punctures. Note now that $\PMap(\Sigma_{n+3}) \times \bZ \cong PB_{n+2} $ \cite[p. 252]{FaMa11}, so Artin groups of $\tilde{C}_n$ virtually embed in $PB_{n+2}$. Hence they are virtually $\RFRp$.

\end{proof}

\bibliographystyle{alpha}
\bibliography{references.bib}

\end{document}